\documentclass[12pt]{amsart}
\usepackage{amsfonts,amssymb,amscd,amsmath,amsthm,enumerate,verbatim,calc}
\theoremstyle{plain}
\newtheorem{Theorem}{Theorem}[section]
\newtheorem{Main Theorem}{Main Theorem}
\newtheorem{Lemma}[Theorem]{Lemma}
\newtheorem{Corollary}[Theorem]{Corollary} 
\newtheorem{Proposition}[Theorem]{Proposition}
\newtheorem{Question}{Question}
\newtheorem{Problem}{Problem}

\theoremstyle{definition}
\newtheorem{Definition}[Theorem]{Definition}
\newtheorem{Example}[Theorem]{Example}
\newtheorem{Remark}[Theorem]{Remark}

\newtheorem{Discussion}[Theorem]{Discussion}
\theoremstyle{remark}

\newtheorem*{chunk*}{}

\numberwithin{equation}{Theorem}

\newcommand{\mlabel}[1]%
{\mbox{}\marginpar{\raggedleft\hspace{0pt}{\rm\ttfamily#1}}\label{#1}}

\newcommand{\Ker}{\operatorname{Ker}}

\newcommand{\Spec}{\operatorname{Spec}}
\newcommand{\fm}{{\mathfrak m}}

\newcommand{\fp}{{\mathfrak p}}

\newcommand{\fq}{{\mathfrak q}}

\newcounter{hours}\newcounter{minutes}

\newcommand{\excise}[1]{}

\begin{document}
\title[F-coherent rings with applications to tight closure theory]
{F-coherent rings with applications to tight closure theory}
\author[K.Shimomoto]{Kazuma Shimomoto}

\email{shimomotokazuma@gmail.com}



%\subjclass{13}
%\subjclass[2000]{Primary 13-XX}
%\subjclass[2000]{Primary ; Secondary}
%\date{\today \, (\printtime)}
%\date{\today}
\maketitle

\begin{abstract}
The aim of this paper is to introduce a new class of Noetherian rings of positive characteristic in terms of perfect closures and study their basic properties. If the perfect closure of a Noetherian ring is coherent, we call it an $F$-coherent ring. Some interesting applications are given in connection with tight closure theory. In particular, we discuss relationships between $F$-coherent rings and $F$-pure, $F$-regular, and $F$-injective rings. The final section discusses how the coherent property effects the behavior of tight closure for general perfect rings.
\end{abstract}

\bigskip

\section {Introduction}

In the present paper, all rings are assumed to be commutative and unitary. Recall that a ring is $\textit{coherent}$ if every finitely generated ideal is finitely presented. This condition is automatic for all Noetherian rings. We shall consider the following problem in connection with tight closure theory.

\begin{Problem}
Let $R$ be a Noetherian ring of positive characteristic and let $R^{\infty}$ be its perfect closure. Then what is characteristic of $R$ if $R^{\infty}$ is a coherent ring?
\end{Problem}

Tight closure theory was created in the mid 80's by Hochster and Huneke as a powerful tool with applications to many outstanding questions regarding rings of positive characteristic. In fact, a modification of the above problem has a connection with the problem whether tight closure commutes with localization, or not. A typical attempt to establish the localization problem for tight closure was made by its comparison with plus closure. More precisely, if it were true that $R^{+}$ (this notation will be explained later) is coherent, then it would follow that tight closure commutes with localization. Unfortunately, it turns out that this is false, as was shown by the recent work of Brenner and Monsky~\cite{BM}. Except the case when $R$ is a field, the perfect ring $R^{\infty}$ is almost never Noetherian, but it can be better-behaved than $R$ itself. As we shall see below, perfect coherent rings of certain type force their finitely generated ideals to be tightly closed. This result relies on the flatness of the Frobenius map on perfect rings. Also, there are non-trivial cases when the perfect closures of rings, much smaller than $R^{+}$, are not coherent. In this respect, it would be adequate to call a Noetherian ring of positive characteristic $\textit{F-coherent}$ if its perfect closure is coherent. Especially, a subclass of affine semigroup rings forms an interesting class of $F$-coherent rings.

These observations also raise the following question. Let $(R,\fm)$ be a local Noetherian ring of positive characteristic. Then under which condition of $R$ is a system of parameters of $R$ a regular sequence on the perfect closure $R^{\infty}$, or is $R^{\infty}$ a big Cohen-Macaulay $R$-algebra?

We prove by using a valuation, that if $R$ is $F$-coherent, then every system of parameters of $R$ is a regular sequence on $R^{\infty}$. This result is seen to be related to the recent work by Roberts, Singh, and Srinivas. Quite generally, if $R$ is any complete local domain of positive characteristic, then $R^{\infty}$ is almost Cohen-Macaulay (\cite{RSS} for this terminology). Note, however, that even if $R^{\infty}$ is almost Cohen-Macaulay, it does not necessarily follow that $R$ is Cohen-Macaulay. Finally, we mention that homological aspects of perfect closures, as well as absolute integral closures of Noetherian domains are studied by Asgharzadeh~\cite{Asg}

\section{Preliminaries on tight closure}

Let $R$ be a ring of characteristic $p>0$. Let $F^{e}_{R}:R \to R^{(e)}$ be the $e$-th iterated Frobenius map in which $R^{(e)}=R$ as a left $R$-module, while the right $R$-module structure of $R^{(e)}$ is provided by $a \cdot r=a F^{e}(r)$. The $\textit{Frobenius}$ ($\textit{Peskine-Szpiro}$) $\textit{functor}$ $\mathbf{F}_{R}^{e}(-)$ is defined by $\mathbf{F}_{R}^{e}(M)=R^{(e)} \otimes_{R} M$ for any $R$-module $M$.

Let us denote by $R^{0}$ the complement of the set of all minimal primes of $R$ and let $I$ be an ideal. Then we define $I^{[q]}$ as an ideal generated by $q=p^{e}$-powers of all elements of $I$. Then the $\textit{tight closure}$ $I^{*}$ of $I$ is the set of $x \in R$ such that there exists $c \in R^{0}$ for which $cx^{q} \in I^{[q]}$ for $q=p^{e} \gg 0$. It is easy to see that $I^{*}$ is an ideal containing $I$. An ideal $I$ is $\textit{tightly closed}$ if $I=I^{*}$. A Noetherian ring $R$ is called $\textit{weakly F-regular}$ if every ideal of $R$ is tightly closed. $R$ is called $\textit{F-regular}$ if every localization of $R$ is weakly $F$-regular.

The $\textit{perfect closure}$ of a ring $R$ is defined by adjoining to $R_{\mathrm{red}}$ all higher $p$-power roots of all elements of $R_{\mathrm{red}}$ and denote it by $R^{\infty}$. We also use the following notation. For a ring $R$, we let $R^{1/q}$ denote $(R_{\mathrm{red}})^{1/q}$. The $\textit{Frobenius closure}$ of $I$ is defined as $I^{F}=I R^{\infty} \cap R$. An alternative definition is that $x \in I^{F}$ if and only if $x^{q} \in I^{[q]}$ for all $q=p^{e} \gg 0$.  It is obviously true that $I^{F} \subseteq I^{*}$. A Noetherian ring $R$ is called $\textit{F-pure}$ if $I=I^{F}$ for every ideal $I$ of $R$. If $(R,\fm)$ is a local Noetherian ring of characteristic $p > 0$, then the Frobenius map on $R$ naturally induces a map on the \v Cech complex of $R$, hence it gives the Frobenius action on the local cohomology modules $H^{k}_{\fm}(R) \to H^{k}_{\fm}(R^{(1)})$. $(R,\fm)$ is $\textit{F-injective}$ if the map $H^{k}_{\fm}(R) \to H^{k}_{\fm}(R^{(1)})$ is injective for all $0 \le k \le \dim R$. If $(R,\fm)$ is Cohen-Macaulay, then $R$ is $F$-injective if and only if $I=I^{F}$ for any parameter ideal $I$ of $R$. Here $I$ is a $\textit{parameter ideal}$ if $I$ is generated by a system of parameters. Then it holds that $F$-purity implies $F$-injectivity.

The $\textit{absolute integral closure}$ of an integral domain $R$ is the integral closure of $R$ in an algebraic closure of the field of fractions of $R$ and denote it by $R^{+}$. Then clearly $R^{\infty} \subseteq R^{+}$. 

The most part of tight closure theory is concerned with Noetherian rings, while the definition of tight closure itself, at least, makes sense for any ring of characteristic $p > 0$. We will use only the definition to derive the property of coherence on certain perfect rings.

\section{$F$-coherent rings}

A ring $R$ is $\textit{coherent}$ if its every finitely generated ideal is finitely presented. An equivalent definition for $R$ to be coherent is that both $(0:a)$ and $I \cap J$ are finitely generated ideals for any $a \in R$ and finitely generated ideals $I, J \subseteq R$. Another equivalent definition is that every finitely generated submodule of every finitely presented module over $R$ is finitely presented over $R$ (\cite{Glaz}, Theorem 2.3.2 for these definitions). We will use them interchangeably. Let us begin with the definition of $F$-coherent rings and establish some basic properties.

\begin{Definition}
Let $R$ be a Noetherian ring of characteristic $p>0$. Then we say that $R$ is $F$-$\textit{coherent}$ if the perfect closure $R^{\infty}$ is coherent.
\end{Definition}

A simple observation shows that $R$ is $F$-coherent if and only if $R_{\mathrm{red}}$ is $F$-coherent. In particular, we may assume that $R$ is always reduced. In the following, when we say a $\textit{flat colimit}$ of commutative rings, it means a direct system $\{R_{\alpha}\}_{\alpha \in \Lambda}$ such that, whenever $\alpha < \beta$, the transition map $R_{\alpha} \to R_{\beta}$ is flat in the usual sense.

\begin{Proposition}
\label{coherent-1}
Let $R$ be a Noetherian ring of characteristic $p>0$. 

\begin{enumerate}
\item[$\mathrm{(1)}$]
Any regular ring is $F$-coherent.

\item[$\mathrm{(2)}$]
Let $S$ be a multiplicative subset of an $F$-coherent ring $R$. Then the localization $S^{-1} R$ is $F$-coherent as well.

\item[$\mathrm{(3)}$]
If $U=\{\fp \in \Spec R~|~R_{\fp}~\mbox{is $F$-coherent}\}$ is constructible, then $U$ is a non-empty Zariski open subset.

\item[$\mathrm{(4)}$]
Let $R \subseteq S$ be a faithfully flat extension. Then if $S$ is $F$-coherent, so is $R$.
\end{enumerate}
\end{Proposition}

\begin{proof}
For $(1)$, we note that $R \to R^{1/p} \to \cdots \to R^{1/p^e} \to \cdots$ is a direct system with flat transition maps by using a theorem of Kunz~\cite{Ku}, so that $R^{\infty}$ is coherent since the flat colimit of coherent rings is coherent.

For $(2)$, let $S$ be a multiplicative subset of $R$. We show that there is an isomorphism: $(S^{-1} R)^{\infty} \simeq S^{-1}(R^{\infty})$. Since the problem depends only on $R_{\mathrm{red}}$, we may assume that $R$ is already reduced. Then the natural map $R \to R^{\infty}$ extends to an injection $S^{-1}R \to S^{-1}(R^{\infty})$. Since this map is purely inseparable and the localization of a perfect ring is again perfect, the required isomorphism follows. On the other hand, the property of coherence is stable under localization, so the isomorphism implies that $(S^{-1}R)^{\infty}$ is also coherent.

For $(3)$, we pick $\fp, \fq \in \Spec R$ such that $\fp \subseteq \fq$ and $\fq \in U$. Then since the localization of an $F$-coherent ring is $F$-coherent, it follows that $\fp \in U$, which implies that $U$ is stable under generization. Hence $U$ is open by assumption. That $U$ is non-empty follows from the fact that the total ring of fractions of $R_{\mathrm{red}}$ is the finite product of fields, which is obviously $F$-coherent.

For $(4)$, we prove the following fact: Let $R \subseteq S$ be a faithfully flat extension and let $M$ be an $R$-module. Then $M$ is finitely generated over $R$ if and only if $M \otimes_{R} S$ is finitely generated over $S$. The ``only if'' part is obvious. Conversely, assume that the finite set of elements $s_{1},\ldots,s_{n}$ of $M$ generates the $S$-module $M \otimes_{R} S$ under the inclusion $M \to M \otimes_{R} S$. Let $N$ be an $R$-submodule of $M$ that is generated by $s_{1},\ldots,s_{n}$. Then the exact sequence: $0 \to N \to M \to M/N \to 0$ gives that $(M/N) \otimes_{R} S=0$, hence $M/N=0$ by faithful flatness of $R \to S$ and this implies that $M$ is generated by $s_{1},\ldots,s_{n}$. We may then apply this fact together with (\cite{Mat}, Theorem 7.4) to conclude that $(0:_{R}a)$ and $I \cap J$ are both finitely generated for finitely generated ideals $I$, $J$ and $a \in R$. So if $S$ is $F$-coherent, so is $R$.
\end{proof}

It is not clear as to what to expect on the topology of the $F$-coherent locus of $\Spec R$ in general. For example, if $R$ is an excellent domain, then the $F$-coherent locus contains an open subset, since the regular locus is open.

\begin{Remark}
As a natural extension of the argument used to show that a regular ring is $F$-coherent, it can be shown that the flat colimit of $F$-coherent rings is $F$-coherent, if the colimit is Noetherian. In fact, this follows easily from the fact that the perfect closure of a ring $R$ of characteristic $p>0$ is obtained as the colimit: 
$$
\begin{CD}
\varinjlim \{R @>F>> R @>F>> R @>F>> \cdots \},
\end{CD}
$$
where the map is the Frobenius map, together with the fact that if the map from one direct system to another system is flat, its colimit is also flat.
\end{Remark}

Although the next corollary is simple, it serves as a useful way for producing sufficiently many $F$-coherent rings.

\begin{Corollary}
\label{coherent-3}
Let $R \to S$ be a purely inseparable extension of Noetherian rings. Then $R$ is $F$-coherent if and only if $S$ is so. In particular, if $R$ is a purely inseparable extension, or subextension of a polynomial algebra over a field of characteristic $p>0$, it is $F$-coherent.
\end{Corollary}

\begin{proof}
The proof of this corollary is immediate from the definition.
\end{proof}

\begin{Example}
\label{coherent-2}
Affine semigroup rings provide non-trivial (non-regular) examples of $F$-coherent rings. They even include rings which are neither Cohen-Macaulay, nor normal. Let $k$ be any field of characteristic $2$ and let $R=k[x^{4},x^{3}y,xy^{3},y^{4}]$. Then $(x^{3}y)^{2}/x^{4}=x^{2}y^{2}$, which is integral over $R$, but is not in $R$. The failure of Cohen-Macaulay property is easy to see. On the other hand, we have $k[x^{4},y^{4}] \subseteq R \subseteq k[x,y]$, which is a tower of purely inseparable extensions, hence $R$ is $F$-coherent. Later on, we prove an easy-to-use criterion in terms of Cohen-Macaulay property to see that $R$ is not $F$-pure. Is the normalization $k[x^{4},x^{3}y,x^{2}y^{2},xy^{3},y^{4}]$ of $R$ an $F$-coherent ring?
\end{Example}

\begin{Discussion}
We shall use a valuative method for deriving various properties for $F$-coherent rings. Let $R$ be any Noetherian domain and let $P$ be its prime ideal. Then there exists a discrete valuation domain $(V,tV)$ such that $R \subseteq V \subseteq K$ and $P=R \cap tV$, where $K$ is the field of fractions of $R$. If $R \subseteq S$ is an integral extension domain, then the valuation $v$ attached to $V$ extends to $S$ with $\mathbb{Z}$ or $\mathbb{Q}$ as its value group. In particular, for some non-zero element $a \in R$, we get $v(a^{1/r})=\frac{1}{r} \cdot v(a)$. In what follows, we shall say that $(V,tV)$ as above is $\textit{attached}$ to the pair $(S,P)$ for the brevity of notation.
\end{Discussion}

A question which naturally arises is, of course, to ask how $F$-coherent rings are related with those rings that are studied in tight closure theory. We occasionally use some basic facts on coherent modules, for which we refer to~\cite{Glaz}.

\begin{Theorem}
\label{coherent-4}
Let $R$ be a reduced Noetherian ring of characteristic $p>0$ with module-finite normalization. If $R$ is $F$-coherent, then $R^{\infty}$ is a normal ring. In particular, for any such an $F$-coherent ring, the normalization $R\ \to \overline{R}$ is purely inseparable.
\end{Theorem}

\begin{proof}
Let $R \to \overline{R}$ be the normalization map. Since this map is module-finite, there exists a non-zero divisor $r \in R$ such that $r \cdot \overline{R} \subseteq R$. By iterating the Frobenius map, we get the commutative diagram
$$
\begin{CD}
R @>>> R^{1/p} @>>>  R^{1/p^2} @>>> \cdots \\
@AAA @AAA @AAA \\
r \overline{R} @>>> r^{1/p}(\overline{R})^{1/p} @>>>  r^{1/p^2}(\overline{R})^{1/p^2} @>>> \cdots \\
\end{CD}
$$
and taking the direct limit, we have $r^{1/q}(\overline{R})^{\infty} \subseteq R^{\infty}$ for any $q=p^{e}$. What we need to show is $(\overline{R})^{\infty}=R^{\infty}$. Let $R^{\infty} \subseteq T \subseteq (\overline{R})^{\infty}$ be any module-finite extension. Then 
$$
\begin{CD}
0 @>>> R^{\infty} @>>> T @>>> T/R^{\infty} @>>> 0 \\
\end{CD}
$$
is a short exact sequence of coherent $R^{\infty}$-modules (\cite{Glaz}, Theorem 2.2.1). Next, let $N:=T/R^{\infty}=R^{\infty} \cdot u_{1}+\cdots+R^{\infty} \cdot u_{k}$ and define the $R^{\infty}$-module map $\phi:R^{\infty} \to N^{\oplus k}$ by $\phi(a)=(au_{1},\ldots,au_{k})$. Then we see that $\Ker \phi=(0:_{R^{\infty}}N)$ and $r^{1/q} \in \Ker \phi$ for all $q=p^{e}$. But since $\phi$ is a map of coherent $R^{\infty}$-modules, $\Ker \phi$ is finitely generated in $R^{\infty}$ by (\cite{Glaz}, Corollary 2.2.2).

Assume that $\Ker \phi  \ne R^{\infty}$. Then choosing a minimal prime $P$ of $R^{\infty}$ that is contained in a maximal ideal containing $\Ker \phi$, we consider the natural map $\Psi:R^{\infty} \to R^{\infty}/P$. Then $J=\Psi(\Ker \phi)$ is a finitely generated proper ideal and for $r \in R$ as above, we see that $s=\Psi(r)$ is non-zero because $r \notin P$. $A'=R/(R \cap P) \to A=R^{\infty}/P$ is an integral extension and $A'$ is Noetherian. Extend $A'$ to its module-finite extension $B$ such that $B \subseteq A$ and the generators of $J$ are contained in $B$. Let $(V,tV)$ be a discrete valuation ring attached to $(B,Q)$ for $Q \in \Spec B$, $J \subseteq Q$. Let $v$ be the extended valuation to $A$. But then we get $v(s^{1/q})=\frac{1}{q} \cdot v(s) \to 0$ $(q \to \infty)$ and $\inf\{v(a)~|~a \in J\} \ge v(t)=1$, which are not compatible with each other, as we know $r^{1/q} \cdot R^{\infty} \subseteq \Ker \phi$. Hence $R^{\infty}=T$, which proves that $R^{\infty}=(\overline{R})^{\infty}$.
\end{proof}

\begin{Corollary}
\label{coherent-5}
Let $R$ be a one-dimensional reduced excellent ring of characteristic $p>0$. Then $R$ is $F$-coherent if and only if the normalization of $R$ is purely inseparable over $R$.
\end{Corollary}

As we will see later, this result gives a geometric interpretation of one-dimensional singularities of ``$F$-coherent type'' in characteristic zero. Roughly speaking, if the normalization separates a single singular point of a variety into at least two smooth points, then it is not of $F$-coherent type.

\begin{Example}
Let $R=k[x^{4},x^{2}y,xy^{2},y^{4}]$ for a field $k$ of characteristic $p > 2$. Then we see that $(xy^{2})^{2}/y^{4}=x^{2}$ is in the normalization of $R$. Now assume that $R$ is $F$-coherent. In view of Theorem~\ref{coherent-4}, $x^{2}$ is purely inseparable over $R$. But then $(x^{4})^{n}=(x^{2})^{p^k}$ for some $k > 0$ and thus $2n=p^k$, which cannot be the case since $p > 2$. Therefore, $R$ is not $F$-coherent. However, if $k$ has characteristic 2, then $R$ is $F$-coherent. Finally, the normalization of $R$ is just $k[x,y]$ for any field $k$, so it is $F$-coherent. 
\end{Example}

\begin{Theorem}
\label{coherent-6}
Let $(R,\fm)$ be a reduced local ring which is a residue class ring of a Gorenstein local ring. Then if $R$ is an $F$-coherent ring, every system of parameters of $R$ is a regular sequence on $R^{\infty}$.
\end{Theorem}

\begin{proof}
The claim is that $R^{\infty}$ is a big Cohen-Macaulay $R$-algebra. First of all, we need to show that every system of parameters of $R$ forms an almost regular sequence on $R^{\infty}$ (\cite{RSS}, Proposition 1.4 for the case when $R$ is a domain). Let $d=\dim R$ and $x_{1},\ldots,x_{d}$ be a system of parameters. Then by (\cite{BH}, Corollary 8.1.4), there is a non-nilpotent element $c \in R$ for which 
$$
c \cdot \big((x_{1}^{p^e},\ldots,x_{i}^{p^e}):_{R} x_{i+1}^{p^e}\big) \subseteq (x_{1}^{p^e},\ldots,x_{i}^{p^e})
$$
for $0 \le i \le d-1$ and any integer $e > 0$. Since $R^{1/p^e} \simeq R$ under the $e$-th Frobenius map, we have $c^{1/p^e} \cdot \big((x_{1},\ldots,x_{i}):_{R^{1/p^e}} x_{i+1}\big) \subseteq (x_{1},\ldots,x_{i})R^{1/p^e}$. Since $e > 0$ is arbitrary, we have
$$
c^{1/q} \cdot \big((x_{1},\ldots,x_{i}):_{R^{\infty}} x_{i+1}\big) \subseteq (x_{1},\ldots,x_{i})R^{\infty}
$$
for all $q=p^e$. For a contradiction, consider the multiplication map:
$$
\begin{CD}
R^{\infty}/(x_{1},\ldots,x_{i})R^{\infty} @>x_{i+1}>> R^{\infty}/(x_{1},\ldots,x_{i})R^{\infty}
\end{CD}
$$
and assume that for some $i>0$, the kernel of the above map contains a non-zero cyclic $R^{\infty}$-submodule $N$. Since $R^{\infty}/(x_{1},\ldots,x_{i})R^{\infty}$ is finitely presented over $R^{\infty}$, it is coherent by (\cite{Glaz}, Theorem 2.3.2), and thus we find that $N$ is a finitely presented cyclic $R^{\infty}$-module. So we may write it as $R^{\infty}/J$ for some finitely generated ideal $J$. Note that $J$ is not a unit ideal.

Knowing that $c^{1/q} \cdot R^{\infty} \subseteq J$, $R^{\infty}$ is local, and $c$ is not nilpotent, we can find a minimal prime $P$ of $R^{\infty}$ such that $c \notin P$ and $(J+P)$ is a proper ideal. After reducing $R^{\infty}$ modulo $P$, we may argue as previously to get a contradiction by finding some optimal valuation on $R^{\infty}/P$. Hence $x_{i+1}$ must be a non-zero divisor of $R^{\infty}/(x_{1},\ldots,x_{i})R^{\infty}$. As our initial choice of a system of parameters was arbitrary, we proved the theorem.
\end{proof}

For example, if $R$ is a complete local domain and $F$-pure, but not Cohen-Macaulay, then in view of Theorem~\ref{coherent-6}, a system of parameters of $R$ is not a regular sequence on $R^{\infty}$. Hence $R$ is not $F$-coherent. Rings of this type abound. Example~\ref{coherent-2} shows that even if $R^{\infty}$ is Cohen-Macaulay, $R$ is not so in general. We will explore this issue later again. The following proposition is taken from (\cite{AH}, Proposition 2.3) with a slight modification.

\begin{Proposition}[Criterion for $F$-purity and $F$-regularity]
\label{coherent-7}
Let $R$ be a reduced Noetherian ring of characteristic $p>0$. Suppose that $R$ is $F$-coherent and $I \subseteq R$ is any ideal. Then we have $I^{F}=I^{*}$. In particular, if $R$ is $F$-coherent, $F$-purity is equivalent to $F$-regularity.
\end{Proposition}

\begin{proof}
We first recall that $I^{F} \subseteq I^{*}$. For a contradiction, let $u \in R$ be such that $u \in I^{*}$, but $u \notin I^{F}$. Then coherence implies that $J:=IR^{\infty}:_{R^{\infty}} u$ is a finitely generated non-unit ideal, and we have $c u^{q} \in I^{[q]}$ for $q=p^{e} \gg 0$ and $c \in R^{0}$ by our assumption. Hence $c^{1/q}u \in IR^{1/q} \subseteq IR^{\infty}$, or equivalently $c^{1/q} \cdot R^{\infty} \subseteq J$ for $q=p^{e} \gg 0$. Since $c \in R^{0}$, we can find a minimal prime $P$ of $R^{\infty}$ such that $c \notin P$ and $(J+P)$ is a proper ideal. Then we may argue as previously to conclude that $I^{F}=I^{*}$.

Suppose next that $R$ is $F$-coherent and $F$-pure. Then the above discussion gives that $I=I^{F}=I^{*}$ for any ideal $I$, which implies that every ideal of $R$ is tightly closed. In order to show that every ideal in every localization of $R$ is tightly closed, it suffices to recall that $F$-coherence is stable under localization, as shown previously. This completes the proof of the proposition.
\end{proof} 

As a consequence from the above proposition, tight closure commutes with localization in any $F$-coherent ring. Namely, we have the following.

\begin{Corollary}
Let $R$ be any reduced $F$-coherent ring of characteristic $p>0$ and let $I \subseteq R$ be any ideal. Then $S^{-1}(I^{*})=(S^{-1}I)^{*}$ for any multiplicative subset $S \subseteq R$.
\end{Corollary}

\begin{Remark}
By the main results of~\cite{BLR}, a positive affine semigroup ring $k[C]$ is $F$-pure if and only if it is normal (see also~\cite{HR}, Theorem 5.33). In view of Proposition~\ref{coherent-7} and the example $k[x^{4},x^{2}y,xy^{2},y^{4}]$ treated above, it is interesting to know whether the normality is equivalent to $F$-coherence for affine semigroup rings. If this is true, it would provide many examples of $F$-coherent rings which are not associated with regular rings, $i.e.,$ they are not purely inseparably related.
\end{Remark}

We state some consequences on the Cohen-Macaulay property for $F$-coherent rings under various conditions.

\begin{Corollary}[Criterion for Cohen-Macaulayness I]
Suppose that $R$ is a reduced $F$-coherent ring of characteristic $p>0$ and satisfies one of the following conditions:

\begin{enumerate}
\item[$\mathrm{(1)}$]
$R$ is a residue class ring of a Cohen-Macaulay ring.

\item[$\mathrm{(2)}$]
$R$ is excellent.
\end{enumerate}

Then if $R$ is $F$-pure, it is Cohen-Macaulay.
\end{Corollary}

\begin{proof}
By the previous proposition, we see that if $R$ is $F$-pure, then it is $F$-regular, so that our assumptions imply that it is Cohen-Macaulay due to (\cite{HH1}, Theorem 4.2 and Proposition 6.27).
\end{proof}

The following corollary is very similar to the above one, but there are minor differences in their assumptions.

\begin{Corollary}[Criterion for Cohen-Macaulayness II]
Let $(R,\fm)$ be a reduced local ring which is a residue class ring of a Gorenstein local ring.  If $R$ is $F$-coherent and $F$-injective, then $R$ is Cohen-Macaulay.
\end{Corollary}

\begin{proof}
We need only show that $H^{k}_{\fm}(R)=0$ for all $k < \dim R$. As was already shown, every system of parameters of $R$ is a regular sequence on $R^{\infty}$, so $H^{k}(R^{\infty})=0$ and the natural map $H^{k}_{\fm}(R) \to H^{k}_{\fm}(R^{\infty})$ is injective for all $k < \dim R$ since $R$ is $F$-injective, which clearly gives $H^{k}_{\fm}(R)=0$ for $k < \dim R$, as desired.
\end{proof}

\begin{Definition}
Let $R \to S$ be a ring extension. Then we say that $R$ is an $\textit{algebra retract}$ if there is a ring homomorphism $\phi:S \to R$ such that the restriction of $\phi$ to $R$ is the identity map on $R$.
\end{Definition}

\begin{Theorem}
Let $R \to S$ be a ring extension of reduced Noetherian rings of characteristic $p>0$ and suppose that $R$ is an algebra retract of $S$ and $S$ is $F$-coherent. Then $R$ is $F$-coherent.
\end{Theorem}

\begin{proof}
Let $\phi:S \to R$ be a retraction map. Then we can extend $\phi$ to a ring homomorphism $\phi_{\infty}:S^{\infty} \to R^{\infty}$ such that the restriction $\phi_{\infty} \vert_{R^{\infty}}$ is the identity map as follows. The composition of ring homomorphisms:
$$
\begin{CD}
S^{1/q} @>\simeq >> S @>\phi>> R @>\simeq>> R^{1/q}, \\
\end{CD}
$$
where the first and the third maps are given by the Frobenius maps, yields a compatible sequence of retraction maps $\phi_{e}$ with $\phi_{0}=\phi$ for every $q=p^e$ and taking its direct limit, we find that $\varinjlim_{e} \phi_{e}$ is the desired map. Now we apply (\cite{Glaz}, Theorem 4.1.5) to conclude that $R^{\infty}$ is coherent, as desired.
\end{proof}

I do not know if the above theorem holds, only assuming that $R$ is a direct summand of $S$ as an $R$-module. The situation in the above theorem is to be seen in the case of affine semigroup rings. Let $k[C]$ be an affine semigroup ring and let $F$ be a face of the cone $\mathbb{R}_{+}C$. Then $k[F \cap C]$ is an algebra retract of $k[C]$. This follows from the fact that $C \backslash F$ is a (prime) ideal. Let $P_{F}$ be an ideal generated by the monomials $X^{c}$ for $c \in C \backslash F$. Then $k[C]=k[C \cap F]+P_{F}$ and clearly $k[C \cap F] \cap P_{F}=0$. Hence $k[C \cap F]$ is an algebra retract of $k[C]$. 

The next proposition shows that we can always reduce to the henselian local rings to study the $F$-coherent property of local rings.

\begin{Proposition}
Let $(R,\fm)$ be a reduced local ring of characteristic $p>0$. Then $R$ is $F$-coherent if and only if the henselization $R^{\mathrm{h}}$ is so.
\end{Proposition}

\begin{proof}
Since $R \to R^{\mathrm{h}}$ is faithfully flat, $R$ is $F$-coherent, if $R^{\mathrm{h}}$ is so. Conversely, assume that $R$ is $F$-coherent. Then we show that $F$-coherence passes to any finite \'etale extension $R \to S$. Since $S$ is \'etale over $R$, the natural ring homomorphism $R^{1/q}\otimes_{R} S  \to S^{1/q}$ is an isomorphism for all $q > 0$. Hence we have $R^{\infty} \otimes_{R} S \simeq S^{\infty}$. Since $R^{\infty} \otimes_{R} S$ is a module-finite free extension of $R^{\infty}$, it is coherent as well. We also proved that the localization of an $F$-coherent ring is $F$-coherent. Therefore, any standard \'etale extension of $R$ is $F$-coherent. Quite generally, we note that if $R \to S \to T$ is a composition of ring homomorphisms such that $R \to T$ is \'etale and $R \to S$ is unramified, then $S \to T$ is \'etale. In particular, it is flat.

Finally, since the flat colimit of coherent rings is coherent, the henselization of $R$ can be constructed as a colimit of various localizations of module-finite \'etale $R$-algebras, we conclude that $R^{\mathrm{h}}$ is $F$-coherent.
\end{proof}

It is worth suggesting the following list of questions for prompting the research in future.

\begin{Question}
Let $R \to S$ be a faithfully flat map of Noetherian rings such that $R$ is $F$-coherent. Then what conditions on the fibers of $R \to S$ are required in order that $S$ is $F$-coherent?
\end{Question}

We do not know if this question is true for the case where $S$ is smooth over $R$, because we do not even know whether the polynomial ring $R[x]$ is $F$-coherent or not, assuming that $R$ is $F$-coherent.

\begin{Question}
Let $R$ be a Noetherian ring and let $t \in R$ be a non-zero divisor. If $R/tR$ is $F$-coherent, then is $R$ also $F$-coherent?
\end{Question}

This question is not obvious, since the property of coherence is not necessarily inherited from the quotient of a ring by a non-zero divisor. As there is an example of hypersurface ring (a quotient of a regular domain modulo a non-zero element) which is $F$-pure, but not $F$-regular, the converse of the above question does not hold.

\begin{Question}
Let $R$ be an $F$-coherent local ring. Then is the Hilbert-Kunz multiplicity of $R$ rational?
\end{Question}

This question is in connection with the fact that the Hilbert-Kunz multiplicity of a local ring having finite $F$-representation type is rational. It is also interesting to ask:

\begin{Question}
Let $R \subseteq S$ be a pure extension such that $S$ is $F$-coherent. Then is $R$ also $F$-coherent?
\end{Question}

As our examples of $F$-coherent rings always came from the regular rings, it is natural to ask:

\begin{Question}
Is there an example of an $F$-coherent ring whose perfect closure does not coincide with the one which comes from a regular ring?
\end{Question}

\begin{Question}
Suppose $R$ is a Noetherian ring of characteristic $p>0$. Then can one find any practical characterization other than the definition that $R$ is $F$-coherent?
\end{Question}

\begin{Remark}
Let us make some observations about the arithmetic deformation of $F$-coherent rings over $\Spec \mathbb{Z}$ in the one-dimension case. First, we consider $\mathbb{Z}[t^{2},t^{3}] \simeq \mathbb{Z}[x,y]/(x^{3}-y^{2})$. Then in characteristic $p>0$, the ring $(\mathbb{Z}/p\mathbb{Z})[t^{2},t^{3}]$ is $F$-coherent, due to Corollary~\ref{coherent-3}. Hence the generic fiber of $\Spec \mathbb{Z}[x,y]/(x^{3}-y^{2}) \to \Spec \mathbb{Z}$ is supposed to be of $F$-coherent type. On the other hand, let us take $\mathbb{Z}[x,y]/(y^{2}-x^{3}-x^{2})$, an ordinary double point. Then $\mathbb{Z}[x,y]/(y^{2}-x^{3}-x^{2}) \simeq \mathbb{Z}[t,t\sqrt{t+1}]$, so we have that $(\mathbb{Z}/p\mathbb{Z})[t,t\sqrt{t+1}] \to (\mathbb{Z}/p\mathbb{Z})[t,\sqrt{t+1}]$ is the normalization map and purely inseparable only when $p=2$. In view of Corollary~\ref{coherent-5}, the property of fibers being $F$-coherent can be isolated in the arithmetic deformation.
\end{Remark}

\section{Tight closure in perfect rings and coherence}

In this section, we show that Brenner and Monsky's recent example~\cite{BM} that tight closure does not commute with localization is related to the coherence of certain perfect rings. While the failure of coherence for the absolute integral closure was already found by Aberbach and Hochster~\cite{AH}, the main idea discussed in this section has the advantage that it can throw some light on the nature of general perfect rings.

We first establish some basic properties of tight closure of finitely generated ideals on certain perfect rings. Let $A$ be a perfect ring, that is, the Frobenius map is bijective on $A$. Then the Frobenius functor $\mathbf{F}_{A}$ is obviously flat.

\begin{Lemma}
\label{perfect}
Let $A$ be a perfect coherent ring of characteristic $p>0$. Suppose that every finitely generated ideal of $A$ satisfies the Krull's intersection property; that is $\bigcap_{n>0} I^{n}=0$. Then every finitely generated ideal of $A$ is tightly closed.
\end{Lemma}

\begin{proof}
Let $I$ be any finitely generated ideal and let $z \in I^{*}$. Then there exists $c \in A^{0}$ such that $cz^{[q]} \in I^{[q]}$ for $q=p^{e} \gg 0$. If we assume $z \notin I$, then $c \in (I^{[q]}:_{A}z^{[q]})$ is a non-zero proper ideal of $A$. By the flatness of $\mathbf{F}_{A}$, we have
$$
(I:_{A}z)^{[q]}=(I^{[q]}:_{A}z^{[q]}),
$$ 
which is finitely generated, due to (\cite{Glaz}, Theorem 2.3.2). Hence
$$
c \in \bigcap_{q>0} (I:_{A}z)^{[q]} \subseteq \bigcap_{q>0} (I:_{A}z)^{q}=0.
$$
But this is a contradiction.
\end{proof}

\begin{Lemma}
\label{Krull}
Let $R$ be any Noetherian domain and let $I$ be its proper ideal. Then $R^{+}$ is $I$-adically separated.
\end{Lemma}

\begin{proof}
If we assume the contrary, there is a non-zero element $c \in \bigcap_{n>0} I^{n}R^{+}$. Extend $R$ to its module-finite extension domain $S$ so that $c \in IS$, and we may choose a prime ideal $P$ of $S$ such that $IS \subseteq P$ and a discrete valuation ring $(V,tV)$ attached to $(S,P)$. Then we must have $c \in \bigcap_{n>0} I^{n}R^{+} \cap V \subseteq \bigcap_{n>0} t^{n}V^{+} \cap V=\bigcap_{n>0} t^{n}V=0$, which is a contradiction.
\end{proof}

The following proposition looks similar to Proposition~\ref{coherent-7}. However, the difference is that Proposition~\ref{coherent-7} addresses consequences on the tight closure of ideals on Noetherian rings, while the following does so on perfect rings rather than Noetherian rings.

\begin{Proposition}
Let $R$ be a Noetherian domain of characteristic $p>0$. Suppose that $S$ is any perfect domain that is integral over $R$. If there is a finitely generated ideal of $S$ that is not tightly closed, then $S$ is not coherent.
\end{Proposition}

\begin{proof}
Note first that $R \subseteq S \subseteq R^{+}$. Let $I$ be any finitely generated ideal of $S$. Then extend $R$ to its module-finite extension domain $R'$ such that the generators of $I$ are contained in the Noetherian domain $R'$. Then applying Lemma \ref{Krull}, we obtain $\bigcap_{n>0} I^{n}S \subseteq \bigcap_{n>0} I^{n}R^{+}=0$. The proposition now follows from Lemma~\ref{perfect}.
\end{proof}

We end this section with the following example to use the results in this section.

\begin{Example}[Brenner-Monsky]
Recently, an example of a Noetherian domain $R$ of characteristic 2 together with an ideal $I$ such that $I^{+} \ne I^{*}$, was constructed. Here $I^{+}=I R^{+} \cap R$ is the plus closure of $I$. Let $L=\overline{(\mathbb{Z}/2\mathbb{Z})}$, the algebraic closure of $\mathbb{Z}/2\mathbb{Z}$, and
$$
R:=R_{t}=L[x,y,z,t]/(g_{t}),
$$
where $g_{t}=z^{4}+xyz^{2}+x^{3}z+y^{3}z+tx^{2}y^{2}$. Let also $S=L[t]-\{0\}$. The ring $R$ is a three-dimensional normal hypersurface ring. Then they show that for an ideal $I=(x^{4},y^{4},z^{4})$, $y^{3}z^{3} \in (S^{-1}I)^{*}$, but $y^{3}z^{3}$ is not in $S^{-1} (I^{*})$. Quite generally, it is known that $I^{+} \subseteq I^{*}$ (\cite{BH}, Remarks 10.1.6). By definition, plus closure commutes with localization, so that $I^{+} \ne I^{*}$. Their example, in fact, shows that $I^{+} \subsetneq I^{*}$ and thus $R$ is not $F$-coherent. Assume $R^{+}$ is coherent. Then we have by Lemma \ref{perfect} that $IR^{+}=(IR^{+})^{*}$. Therefore:
$$
I^{+}=IR^{+} \cap R \subseteq I^{*}R^{+} \cap R \subseteq (IR^{+})^{*} \cap R=I^{+}, 
$$
which yields a contradiction $I^{*} \subseteq I^{*} R^{+} \cap R=I^{+}$. Here, we used the fact that the tight closure persists in $R^{+}$ from $R$. Hence this shows that $R^{+}$ is not coherent. Of course, this fact follows easily from (\cite{AH}, Proposition 2.3), while our argument uses tight closure on $R^{+}$ directly.
\end{Example}

\end{document}